\documentclass[11pt]{article}
\usepackage[latin1]{inputenc}
\usepackage[english]{babel}
\usepackage{amsthm}
\usepackage{amsmath,amsfonts,amscd,amssymb}
\numberwithin{equation}{section}
\usepackage[shortlabels]{enumitem}

\usepackage{makeidx}

\usepackage{color}

\theoremstyle{plain}
  \newtheorem{theorem}{Theorem}[section]
  
  \newtheorem{corollary}[theorem]{Corollary}
  \newtheorem{lemma}[theorem]{Lemma}
  \newtheorem{proposition}[theorem]{Proposition}
\theoremstyle{definition}
  \newtheorem{definition}[theorem]{Definition}
  \newtheorem{example}[theorem]{Example}
\theoremstyle{remark}
  \newtheorem{remark}[theorem]{Remark}

\let\noi=\noindent

\newcommand{\Q}{\mathbb{Q}}
\newcommand{\Z}{\mathbb{Z}}
\newcommand{\N}{\mathbb{N}}

\newcommand{\g}{\Gamma}

\newcommand{\Ro}{\mathcal{R}}
\newcommand{\D}{\mathcal{D}}
\makeindex
\title{On common extensions of valued fields}
\author{W. Mahboub\\
Lebanese International University\\
Mouseitbeh - PO Box: 146404\\
Mazraa, Beirut, Lebanon.
\and
 A. Mansour\\
Lebanese International University\\
Mouseitbeh - PO Box: 146404\\
Mazraa, Beirut, Lebanon.
\and
M. Spivakovsky\\
Institut de Math\'ematiques de Toulouse\\
UMR 5219 du CNRS,\\
Universit\'e Paul Sabatier\\
118, route de Narbonne\\
31062 Toulouse cedex 9, France and\\
UMI CNRS 2001 LaSol, UNAM.\\email:
mark.spivakovsky@math.univ-toulouse.fr }

\begin{document}
\maketitle
\section{Introduction}

Let $(K,v)$ be a valued field, $\bar{K}$ an algebraic closure of $K$ and $\bar{v}$ an extension of $v$ to $\bar{K}$. Let $X$ be a transcendental element over $K$. With the aim of giving a characterization of residual transcendental extensions of $v$ to $K(X)$ \cite{AZ},  V. Alexandru and A. Zaharescu introduce the notion of "a minimal pair of definition". They prove that  describing all such extensions is equivalent to describing all the minimal pairs $(a,\delta)\in\bar{K}\times\g_{\bar{v}}$ (see Definition \ref{minimalpairofdefinition} below), where $\g_{\bar{v}}$ is the value group of $\bar{v}$.\\

In \cite{APZ}, V. Alexandru, N. Popescu and A. Zaharescu investigate which pairs $(a,\delta)\in\bar{K}\times \g_{\bar{v}}$ are minimal pairs. Given an extension $w$ of $K$ to $K(X)$, the authors define a common extension of $\bar{v}$ and $w$ to $\bar{K}(X)$, and they prove that there exists an integer, denoted by $[K:w]$, depending only on $v$ and $w$, and that the number of common extensions $\bar{w}$ is less than or equal to $[K:w]$ (see \cite{APZ}, Corollary 2.3 and its proof).\\

Another way of understanding the extensions $w$ of $v$ to $K(X)$ is via the theory of key polynomials introduced by S. Mac Lane (see \cite{ML1} and \cite{ML2}). The theory was introduced by S. Mac Lane in the case of discrete valuations of rank $1$ and generalized by M. Vaqui\'e  to the case of arbitrary valuations. One important difference with the case of discrete rank $1$ valuations
is the presence of limit key polynomials. Another notion of key polynomials was introduced by F. H. Herrera, M. A. Olalla and M. Spivakovsky (see \cite{HOS} or \cite{HMOS}). Yet another new notion of key polynomials was introduced in \cite{DMS} and \cite{J1}. Comparison between these notions are given in \cite{DMS} and \cite{Ma}. For more information about the key polynomial theory and its applications see \cite{Na2}, \cite{CMT}, \cite{Kas}, \cite{Na1}, \cite{San}  and \cite{San2}.\\

It seems that the notions of key polynomials and minimal pairs are closely related. For recent studies on the relation between the two notions, see  \cite{J2} and \cite{V3}. From Theorem 1.1 and Proposition 3.1 of \cite{J2}, one can deduce that given a valuation $w$ which is residue-transcendental there exists a polynomial $Q$ such that $w=w_Q$  and every common extension $\bar{w}$ of
$\bar{v}$ and $w$ can be described by a pair of definition $(a,\delta)$ where $a$ is a root of $Q$ and $\epsilon(Q)$ (see Section \ref{keypolynomials} for the definition of $w_Q$ and $\epsilon(Q)$). As well, one can deduce that such a $Q$ must be the last key polynomial in a complete sequence of key polynomials for $w$ (see Section \ref{keypolynomials} for the definition of a complete sequence of key polynomials).\\

In this paper we study the relation between key polynomials and minimal pairs. We also describe the classification given in Section 3 of \cite{Ku} (``value-transcendental", ``residue-transcendental",  ``valuation-transcendental'' and  ``valuation-algebraic", see Definition \ref{valuation-} below) of all the possible extensions of $v$ from $K$ to $K(X)$  in terms of a complete sequence of key polynomials. Finally, we prove that the extension $w$ is residue-transcendental or value-transcendental if and only if the complete sequence of key polynomials for $w$ has a last element $Q$. In this case every common extension $\bar{w}$ of $\bar{v}$ and $w$ can be described by a pair of definition $(a,\delta)$ where $a$ is a root of $Q$. Moreover, we prove that in the case when the sequence of key polynomials does not admit a limit key polynomial, any root $a$ of $Q$ can be used to define a minimal pair $(a,\delta)$ that defines a common extension $\bar{w}$ of $\bar{v}$ and $w$.\\

In Section \ref{basics} we recall some basic facts about extensions of valuations. In Section \ref{minimalpairs} we give some properties relating common extensions and minimal pairs. Although the results of this section can also be found in \cite{AZ} and \cite{APZ}, we reproduce them here in order to make the paper as self-contained as possible.\\

In Section \ref{keypolynomials} we recall some basic results on key polynomials. We use the construction of a complete set of key polynomials given in \cite{J1} and summarize briefly the main results of \cite{J1} used in the sequel.\\

In Section \ref{kp_mp} we clarify the relation between the notion of key polynomials and the notion of minimal pair. We also describe the  classification of the extension $w$ in terms of the complete sequence of key polynomials of $w$. This is accomplished in Corollaries \ref{vt_kp} and \ref{rt_kpresult} and Theorem \ref{valuation_trans}. We also deduce that the number of common extensions
$\bar{w}$ of $\bar{v}$ and $w$ is less than or equal to the number of roots of the last key polynomial of the sequence (see Corollary \ref{val_trans_numext}).\\

Finally, in Section \ref{common_ext} we study the relation between the roots of two consecutive key polynomials. We show that in the case when the sequence does not contain limit key polynomials, every root $a$ of the last key polynomial in the sequence can be used to construct a common extension $\bar{w}$ of $\bar{v}$ and $w$ (see Corollary \ref{allroots}).

\section{Basics and Notation}\label{basics}

Throughout this paper, we fix a valued field $(K,v)$, an algebraic closure $\bar{K}$ of $K$ and an extension $\bar{v}$ of $v$ to
$\bar{K}$. We also fix a variable $X$ and an extension $w$ of $v$ to $K(X)$.\\

By Proposition 2.1 in \cite{APZ}, there exists a common extension of $\bar{v}$ and $w$ to $\bar{K}(X)$, that is, a valuation
$\bar{w}$ on $\bar{K}(X)$ that is equal to $\bar{v}$ on $\bar{K}$ and to $w$ on $K(X)$.\\

Let $\bar{w}$ be a common extension of $w$ and $\bar{v}$ to $\bar{K}(X)$.\\

We denote by $\g_{v},\ \g_{\bar{v}},\ \g_{w}$ and $\g_{\bar{w}}$ the respective value groups of $v,\ \bar{v},\ w$ and $\bar{w}$. We have natural embeddings $\g_{v}\subseteq\g_{\bar{v}}\subseteq\g_{\bar{w}}$ and
$\g_{v}\subseteq\g_{w}\subseteq\g_{\bar{w}}$.\\

We denote by $k_{v},\ k_{\bar{v}},\ k_{w}$ and $k_{\bar{w}}$ the respective residue fields of $v,\ \bar{v},\ w$ and $\bar{w}$. We have natural extensions $k_{v}\subseteq k_{\bar{v}}\subseteq k_{\bar{w}}$, and $k_{v}\subseteq k_{w}\subseteq k_{\bar{w}}$.\\

Recall the following definitions from \cite{Ku}:
\begin{definition}\label{valuation-} The extension $w$ of $v$ to $K(X)$ is said to be {\bf valuation-algebraic} if $\frac{\g_w}{\g_v}$ is a torsion group and $k_w$ is algebraic over $k_v$. The extension $w$ is said to be {\bf value-transcendental} if $\frac{\g_w}{\g_v}$ has rational rank 1 and $k_w$ is algebraic over $k_v$. The extension $w$ is said to be {\bf residue-transcendental} if $k_w$ has transcendence degree 1 over $k_v$ and $\frac{\g_w}{\g_v}$ is a torsion group. We will combine the value-transcendental case and the residue-transcendental case by saying that $w$ is {\bf valuation-transcendental} if either $\frac{\g_w}{\g_v}$ has rational rank 1 or $k_w$ has transcendence degree 1 over $k_v$.
\end{definition}
For a polynomial $f(X)\in K[X]$, we define the set $\Ro(f)$ by
$$
\Ro(f):=\left\{a\in\bar{K}\ /\ f(a)=0\right\}.
$$
For an element $y$ we denote by $y^*$ its image in the corresponding residue field.\\

We note that
\begin{equation}
\g_{\bar{v}}=\g_{v}\bigotimes_{\Z}\Q\label{eq:gammavbar}
\end{equation}
and that $k_{\bar{v}}$ is an  algebraic closure of $k_{v}$.\\

We define the set $M_{\bar{w}}:=\{\bar{w}(X-a)\ /\ a\in \bar{K}\}$. \\

Let $a\in\bar{K}$ and let $\delta$ be an element in an ordered group containing $\g_{\bar{v}}$. We define the valuation
$w_{(a,\delta)}$ in the following manner:\\
 For a polynomial $f(X)\in \bar{K}[X]$, write the Taylor expansion of $f$:
$$
f(X)=a_n(X-a)^{n}+\dots+a_1(X-a)+a_0.
$$
Put $w_{(a,\delta)}(f(X))=\inf_{0\leq j\leq n}\{\bar{v}(a_j)+j\delta\}$.\\

We define the set $S_{(a,\delta)}(f):=\{j\in\{1,\dots,n\}\ /\ \bar{v}(a_j)+j\delta=w_{(a,\delta)}(f(X))\}$.\\

We note that if we fix $a\in \bar{K}$ and $\delta=\bar{w}(X-a)$, then for all $f\in \bar{K}[X]$ we have
\begin{equation}
\bar{w}(f(X))\geq w_{(a,\delta)}(f(X)).\label{eq:barw>wadelta}
\end{equation}
If the inequality is strict then $\#S_{(a,\delta)}(f)>1$.

\begin{remark}
In the literature, minimal pairs are defined for residue transcendental valuations. In our case, we give the same definition but more generally for valuation-transcendental valuations.
\end{remark}

\begin{definition}
For a valuation $\mu$ of $\bar{K}(X)$, we say that the pair $(a,\delta)$ is {\bf a pair of definition for the valuation} $\mu$ if
$\mu=w_{(a,\delta)}$. 
\end{definition}

\begin{lemma}\label{pairofdefinition}
Let $(a,\delta)$ be a pair of definition for a valuation $\mu$. Let $b\in \bar{K}$, and let $\delta'$ be an element of an ordered group containing $\g_{\bar{v}}$ as an ordered subgroup. Then $(b,\delta')$ is a pair of definition for $\mu$ if and only if $\delta'=\delta$ and $\bar{v}(a-b)\geq \delta$.
\end{lemma}
\begin{proof}
Suppose first that $\delta'=\delta$ and $\bar{v}(a-b)\geq \delta$. It is sufficient to prove that for all $c\in\bar{K}$, we have
\begin{equation*}
\mu(X-c)=\inf\{\delta,\ \bar{v}(b-c)\}.
\end{equation*}
Take an element $c\in\bar{K}$. We have
\begin{equation*}
\mu(X-c)=\inf\{\delta,\ \bar{v}(a-c)\}.
\end{equation*}
Therefore we have to prove that 
\begin{equation}\label{pairofdefinition_eq}
\inf\{\delta,\ \bar{v}(b-c)\}=\inf\{\delta,\ \bar{v}(a-c)\}.
\end{equation}
We will use the equality $a-c=a-b+b-c$ and the fact that $\bar{v}(a-b)\geq \delta$.\\

If $\delta>\bar{v}(b-c)$ then $\bar{v}(a-b)>\bar{v}(b-c)=\bar{v}(a-c)$, and (\ref{pairofdefinition_eq}) is proved.\\
Otherwise, if $\delta\leq\bar{v}(b-c)$ then $\bar{v}(a-c)\geq \inf\{\bar{v}(a-b),\ \bar{v}(a-c)\}\geq \delta$, and again (\ref{pairofdefinition_eq}) is proved.\\

Conversely, suppose that  $(b,\delta')$ is a pair of definition for $\mu$. We have:
\begin{equation*}
\delta=\mu(X-a)=w_{(b,\delta')}(X-a)=\inf\{\delta',\ \bar{v}(a-b)\},
\end{equation*}
hence $\delta\leq\bar{v}(a-b)$ (this proves the second statement), and $\delta\leq\delta'$. On the other hand, we have
\begin{equation*}
\delta'=\mu(X-b)=w_{(a,\delta)}(X-b)=\inf\{\delta,\ \bar{v}(a-b)\},
\end{equation*}
hence $\delta'\leq\delta$ and we get the desired equality.
\end{proof}

Let $\mu=w_{(a,\delta)}$ for a certain pair $(a,\delta)$ as above. We define the degree of $\mu$
$$
\D(\mu):=min\{[K(b):K]/\ b\in \bar{K}, \mu=w_{(b,\delta)}\}.
$$
\begin{definition}\label{minimalpairofdefinition}
A pair of definition $(a,\delta)$ for $\mu$ is said to be {\bf minimal} if $[K(a):K]=\D(\mu)$. We also say in this case that $(a,\delta)$ is a minimal pair of definition for $\mu$.
\end{definition}

\section{Minimal Pairs}\label{minimalpairs}

In this section, we give some properties of common extensions and minimal pairs.\\
Keep the notation of the previous section.\\

The following result is in \cite{AZ}, Proposition 1.1.

\begin{lemma}\label{rt_lem}
Suppose that $\g_{\bar{v}}=\g_{\bar{w}}$ and let $a\in\bar{K}$. The following conditions are equivalent:
\begin{enumerate}[(a)]
\item $\bar{w}(X-a)=Max ( M_{\bar{w}})$.
\item for each $b\in\bar{K}$ with $\bar{w}(X-a)=\bar{v}(b)$, the element $\left(\frac{X-a}{b}\right)^*$ is transcendental over $k_{\bar{v}}$.
\end{enumerate}
\end{lemma}

\begin{proof}
Suppose (a) is satisfied. Let $b\in\bar{K}$ be such that $\bar{w}(X-a)=\bar{v}(b)$. Put $t=\left(\frac{X-a}{b}\right)^*$. If $t$ is algebraic over $k_{\bar{v}}$ then $t\in k_{\bar{v}}$, since $k_{\bar{v}}$ is algebraically closed. Choose $c\in \bar{K}$ so that $c^*=t$. Now we have $\bar{w}\left(\frac{X-a}{b}-c\right)>0$, that is, $\bar{w}(X-a-cb)>\bar{w}(X-a)$, which is a contradiction.\\

Suppose (b) is satisfied and suppose that there exists $c\in \bar{K}$ such that $\bar{w}(X-c)>\bar{w}(X-a)$. Then
$\bar{w}(X-a+a-c)>\bar{w}(X-a)=\bar{v}(a-c)$. Therefore we have $\bar{w}\left(\frac{X-a}{c-a}-1\right)>0$ and hence
$\left(\frac{X-a}{c-a}\right)^*=1$ in $K_{\bar{v}}$, which is a contradiction.
\end{proof}

\begin{lemma}\label{vt_lem}
Suppose that $\g_{\bar{v}}\subsetneqq\g_{\bar{w}}$. Then $Max ( M_{\bar{w}})$ exists and $Max ( M_{\bar{w}})\notin \g_{\bar{v}}$. Moreover, $Max ( M_{\bar{w}})$ is the unique element in $M_{\bar{w}}$ that does not belong to $\g_{\bar{v}}$.
\end{lemma}
\begin{proof}
Let $f\in \bar{K}[X]$ be such that $\bar{w}(f(X))\notin \g_{\bar{v}}$. Write $f(X)=c\prod\limits_{i=1}^{n}(X-c_i)$. There exists $i$, $1\leq i\leq n$, such that $\bar{w}(X-c_i)\notin \g_{\bar{v}}$.\\
Let $a=c_i$ for such an $i$ and let $\delta=\bar{w}(X-a)$.\\
 
Let $b\in \bar{K}$. If $\bar{w}(X-b)>\delta$. We can write $\bar{w}(X-a+a-b)>\delta$. Then $\bar{v}(a-b)$ is equal to $\delta$ which is impossible.\\

If $\bar{w}(X-b)< \delta$, then $\bar{w}(X-b)=\bar{v}(a-b)\in \g_{\bar{v}}$.

\end{proof}

It is easy to see that $w$ is a residue-transcendental extension of $v$ if and only if $\bar{w}$ is a residue-transcendental extension of $\bar{v}$. Similarly, $w$ is a value-transcendental extension of $v$ if and only if $\bar{w}$ is a value-transcendental extension of $\bar{v}$.\\

Using this fact together with Lemma \ref{rt_lem} and Lemma \ref{vt_lem} we deduce

\begin{proposition}\label{classification}
\begin{enumerate}[(a)]
\item $w$ is residue-transcendental extension of $v$ if and only if $\g_{\bar{v}}=\g_{\bar{w}}$ and $M_{\bar{w}}$ has a maximal element.
\item $w$ is value-transcendental extension of $v$ if and only if $\g_{\bar{v}}\subsetneqq\g_{\bar{w}}$. In this case, again, $M_{\bar{w}}$ has a maximal element.
\item $w$ is a valuation-algebraic extension of $v$ if and only if $M_{\bar{w}}$ does not have a maximal element.
\end{enumerate}
\end{proposition}

\begin{proposition}\label{max_ele} The set $M_{\bar{w}}$ has a maximal element $\delta$ if and only if $\bar{w}=w_{(a,\delta)}$ for some $(a,\delta)\in\bar{K}\times \g_{\bar{w}}$ with $\bar{w}(X-a)=\delta$.
\end{proposition}
\begin{proof}
Suppose that $M_{\bar{w}}$ has a maximal element $\delta$.\\

Suppose first that $\delta\notin\g_{\bar{v}}$. Take an $f(X)\in \bar{K}[X]$. Write the Taylor expansion
$$
f(X)=a_n(X-a)^{n}+\dots+a_0.
$$
For $0\leq \ell_1<\ell_2\leq n$ we must have $\bar{v}(a_{\ell_1})+\ell_{1}\delta\neq \bar{v}(a_{\ell_2})+\ell_{2}\delta$. Hence
$\bar{w}(f(X))=w_{(a,\delta)}(f(X))$.\\

Now suppose that $\delta\in\g_{\bar{v}}$. Then we are in the case when $w$ is a residue-transcendental extension of $v$. We always have the inequality (\ref{eq:barw>wadelta}), so we only need to rule out the {\it strict} inequality in (\ref{eq:barw>wadelta}). Suppose there exists a polynomial $f(X)\in \bar{K}[X]$, such that
\begin{equation}
\bar{w}(f(X))>w_{(a,\delta)}(f(X)).\label{eq:strictinequality}
\end{equation}
Choose a monic polynomial $f$ of minimal degree satisfying the strict inequality (\ref{eq:strictinequality}). Write the Taylor expansion  $f(X)=a_n(X-a)^{n}+a_{n-1}(X-a)^{n-1}+\dots+a_0$ with $a_n=1$.\\

Write $f=\sum\limits_{j\in S_{(a,\delta)}(f)}a_j(X-a)^j+\sum\limits_{j\notin S_{(a,\delta)}(f)}a_j(X-a)^j$. We have\\
\begin{align*}
\bar{w}\left(\sum\limits_{j\in S_{(a,\delta)}(f)}a_j(X-a)^j\right)&\geq
\inf\left\{\bar{w}(f),\bar{w}\left(\sum\limits_{j\notin S_{(a,\delta)}(f)}a_j(X-a)^j\right)\right\}\\
>& w_{(a,\delta)}(f)=w_{(a,\delta)}\left(\sum\limits_{j\in S_{(a,\delta)}(f)}a_j(X-a)^j\right).
\end{align*}
Thus replacing $f(X)$ by $\sum\limits_{j\in S_{(a,\delta)}(f)}a_j(X-a)^j$ does not affect  the strict inequality (\ref{eq:strictinequality}). Hence we may assume that for all $j$, $0\leq j\leq n$, if $a_j\neq 0$ then $j\in S_{(a,\delta)}(f)$. We will make this assumption from now on. In particular, we have $w_{(a,\delta)}(f)=n\delta$.

Now by Lemma \ref{rt_lem} there exists $b$ such that $t=\left(\frac{X-a}{b}\right)^*$ is transcendental over $k_{\bar{v}}$.\\

Using the fact that $\delta=\bar{w}(X-a)=\bar{v}(b)$, we see that $w_{(a,\delta)}\left(\frac{f}{b^n}\right)=n\delta-n\delta=0$. Hence for each $j\in\{0,\dots,n\}$ we have $\bar{w}\left(\frac{a_j}{b^{n-j}}\frac{(X-a)^j}{b^j}\right)=0$, therefore
$\bar{v}\left(\frac{a_j}{b^{n-j}}\right)=0$. We also have
$\bar{w}\left(\frac{f}{b^n}\right)>w_{(a,\delta)}\left(\frac{f}{b^n}\right)=0$.\\

Consider the image of $\frac{f}{b^n}$ in $k_{\bar{w}}$. We have 
$\sum\limits_{j=0}^n\frac{a_j}{b^{n-j}}t^j=0$ with the coefficient of $t^n$ equal to $1$. This contradicts the fact that $t$ is transcendental over $k_{\bar{v}}$.\\

Conversely, suppose that $\bar{w}=w_{(a,\delta)}$ for $(a,\delta)\in \bar{K}\times \g_{\bar{w}}$, with $\bar{w}(X-a)=\delta$. Then for all $b\in \bar{K}$ we have $X-b=X-a+a-b$, hence $\bar{w}(X-b)=\inf \{\delta,\ \bar{v}(a-b)\}$. Therefore $\bar{w}(X-b)\leq \bar{w}(X-a)$.
\end{proof}

\begin{lemma}\label{minimal_pair}
Let $(a,\delta)$ be a pair of definition of $\bar{w}$. For each polynomial $f\in K[X]$ of degree $\deg f<\D(\bar{w})$, we have $w(f(X))=\bar{v}(f(a))$. 
\end{lemma}
\begin{proof}
Take a polynomial $f\in K[X]$ of degree $r<\D(\bar{w})$.
Let $b_1\dots,b_r$ be the roots of $f$ (the $b_i$ need not be distinct).\\
For each $t$, $1\leq t\leq r$, we have $\bar{v}(a-b_t)<\delta$. Indeed, if there existed $t$ such that $\bar{v}(a-b_t)\ge\delta$, by Lemma \ref{pairofdefinition} $(b_t,\delta)$ would be a pair of definition of $\bar{w}$, with $[K(b_t):K]<\D(\bar{w})$. This contradicts the definition of $\D(\bar{w})$.\\
Therefore, for each  $t$, $1\leq t\leq r$, we have $\bar{w}(X-b_t)=\inf\{\delta, \bar{v}(a-b_t)\}=\bar{v}(a-b_t)$.\\

Now write $f(X)=c\prod\limits_{t=1}^r (X-b_t)$.\\

We have $w(f(X))=\bar{w}(f(X))=\bar{v}(c)\sum\limits_{t=1}^r \bar{w}(X-b_t)=\bar{v}(c)\sum\limits_{t=1}^r \bar{v}(a-b_t)=\bar{v}(f(a))$.

\end{proof}

From Proposition \ref{classification} and Proposition \ref{max_ele} we deduce the following Corollary.\\

\begin{corollary} The following conditions are equivalent:

(1) $w$ is valuation-transcendental extension of $v$ to $K(X)$

(2) for every algebraic closure $\bar{K}$ of $K$ and every extension $\bar{v}$ to $\bar{K}$ and a common extension of $w$ and
$\bar{v}$ to $\bar{K}(X)$, there exists $(a,\delta)\in\bar{K}\times\g_{\bar{w}}$ such that $\bar{w}=w_{(a,\delta)}$

(3) there exists an algebraic closure $\bar{K}$ of $K$, an extension $\bar{v}$ to $\bar{K}$ and a common extension of $w$ and
$\bar{v}$ to $\bar{K}(X)$ such that $\bar{w}=w_{(a,\delta)}$ for a certain $(a,\delta)\in\bar{K}\times \g_{\bar{w}}$.

\end{corollary}

\section{Key Polynomials}\label{keypolynomials}
In this section, we introduce the notion of key polynomials of a valuation and study some of their basic properties.\\

\begin{enumerate}
\item For each strictly positive integer $b$, we write $\partial_{b}:=\frac{\partial^{b}}{b!\partial X^{b}}$, the $b${\bf-th formal derivative} with respect to $X$. 
\item For each polynomial $f\in K[X]$, let
$\epsilon(f):=\max\limits _{b\in\mathbb{N}^{\ast}}\left\{\frac{w(f)-w(\partial_{b}P)}{b}\right\}$.
\end{enumerate}
\begin{definition}
Let $Q$ be a monic polynomial in $K[X]$. We say that $Q$ is a {\bf key polynomial} for $w$ if for each polynomial $f$ satisfying
$$
\epsilon(f)\geq\epsilon(Q),
$$
we have $\deg(f)\geq\deg(Q)$.
\end{definition}

For a monic polynomial $Q\in K[X]$, every polynomial $f\in K[X]$ can be written in a unique way as 
\begin{equation}
f=\sum\limits _{j=0}^{s}f_{j}Q^{j},\label{eq:Qexpansion}
\end{equation}
with all the $f_{j}\in K[X]$ of degree strictly less than $\deg(Q)$. We call (\ref{eq:Qexpansion}) the $Q${\bf-expansion}
of $f$. 

\begin{definition}
Let $g=\sum\limits _{j=0}^{s}g_{j}Q^{j}$ be the $Q$-expansion of an element $g\in K[X]\setminus\{0\}$. We put
$$
w_{Q}(g):=\min\limits _{\underset{g_{j}\neq0}{0\leq j\leq s}}w\left(g_{j}Q^{j}\right).
$$
We adopt the convention that $w_Q(0)=\infty$. The mapping $w_{Q}:K[X]\rightarrow\g_w$ is called the {\bf truncation} of
$w$ with respect to $Q$. 
\end{definition}

We have the following Proposition (\cite{J1}, Proposition 2.4 (ii) and Proposition 2.6).

\begin{proposition}
If $Q$ is a key polynomial then $Q$ is irreducible and $w_Q$ is a valuation.\\
\end{proposition}

For each polynomial $f(X)\in K[X]$, let $\delta(\bar{w},f)=max \{\bar{w}(X-a)/\ a\in \bar{K},\ f(a)=0\}$.\\

\begin{proposition}(\cite{J2}, Proposition 3.1)
Let $f(X)\in K[X]$ be a monic polynomial. We have $\delta(\bar{w}, f)=\epsilon(f)$.
\end{proposition}

The quantity $\delta(\bar{w}, f)$ depends only on $w$ and $f$, but not on $\bar{w}$, therefore we will denote it by $\delta(f)$.\\
\begin{definition} Let $\Lambda$ be an ordered set. A set $\{Q_i\}_{i\in\Lambda}$ of key polynomials is said to be {\bf complete} for $w$ if for every $f\in K[X]$ there exists $i\in\Lambda$ such that $w_{Q_i}(f)=w(f)$.
\end{definition}

\begin{remark}\label{prop_key_pol} By \cite{J1}, Theorem 1.1 and its proof, there is a complete set of key polynomials
$\{Q_i\}_{i\in\Lambda}$ for the valuation $w$, having the following properties:

\begin{enumerate}
\item $\Lambda=\bigcup\limits_{j\in I}\Lambda_j$ , with $I=\{0,\dots,N\}$ or $\N$, and for each $j$ we have
$\Lambda_j=\{j\}\cup\vartheta_j$, where $\vartheta_j$ an ordered set without a maximal element, which may be empty.
\item There exists $a\in K$ such that $Q_0=X-a$.
\item For all $j\in I\setminus\{0\}$ we have $j-1<\vartheta_{j-1}<j$.
\item All the polynomials $Q_i$ with $i\in \Lambda_j$ have the same degree, and the polynomials $Q_i$ with $i\in\Lambda_j$ have degree strictly less than the polynomials $Q_{i'}$, $i'\in \Lambda_{j+1}$.
\item For each $i<i'\in\Lambda$ we have $w(Q_i)<w(Q_i')$ and $\epsilon(Q_i)<\epsilon(Q_i')$.
\end{enumerate}
\end{remark}
\medskip

For each $i\in \Lambda$, put $w_{i}=w_{Q_i}$, $\beta_i=w(Q_i)$ and $\epsilon_i=\epsilon(Q_i)$.\\

Even though the set of key polynomials $\{Q_i\}_{i\in \Lambda}$ is not unique, the cardinality of $I$ and the degrees $d_j$ of the key polynomials $Q_j$ for each $j\in I$ are uniquely determined by $w$. As well, the valuations $w_j$ for $j\in I$ are uniquely determined by $w$.
\medskip

\begin{remark}\label{uniqueness} The above uniqueness statements together with conditions 1 and 3 above imply, in particular, that if $i\in\Lambda$ is not the maximal element then the set $\{Q_j\}_{\overset{j\in\Lambda}{j\le i}}$ of key polynomials is not complete (equivalently, $w_i\ne w$). 
\end{remark}

\noi{\bf Notation.} We will by denote by $d(w)$ the degree of $Q_N$ for the maximal element $N$ of $I$ if it exists. If $I=\mathbb N$, we put $d(w)=\infty$.\\

\begin{remark}\label{directconsequence}
A direct consequence of the construction of \cite{J1} is that for all $i\in\Lambda_j$ the value group $\g_i$ of $w_i$ is equal to
$\g_v+\beta_0\Z+\dots+\beta_j\Z$.\\
\end{remark}

Also by construction, $\Lambda$ has a maximal element if and only if the following two conditions hold:
\begin{enumerate}
\item the set $I=\{0,\dots,N\}$ is finite and
\item $\Lambda_n=\{N\}$.
\end{enumerate}

Keep the above notation. Take an element $i\in\Lambda$.

\begin{definition} We say that $Q_i$ is a {\bf limit key polynomial} if the following two conditions hold:
\begin{enumerate}
\item $i\in I\setminus\{0\}$
\item $\vartheta_{i-1}\ne\emptyset$.
\end{enumerate}
\end{definition}

\begin{proposition}\label{kp_form}
There exists a complete set $\{Q_i\}_{i\in \Lambda}$ of key polynomials having the following additional property. For every $j\in I$ such that $\Lambda_{j+1}\ne\emptyset$ and $\vartheta_j=\emptyset$ (in other words, $Q_{j+1}$ is not a limit key polynomial), we can write $Q_{j+1}=q_nQ_j^{n}+\dots+q_tQ_j^t+\dots+q_0$, with $q_n=1$, $\deg q_t<d_j$ and $w(q_t)+t\beta_j=n\beta_j$ for each $t$, $0\leq t\leq n-1$.\\
\end{proposition}

For a proof of Proposition \ref{kp_form}, one can imitate the proof of Theorem 9.4 \cite{ML2} or Theorem 1.11 \cite{V}.
\medskip

\noi{\bf In the sequel, we will always choose a complete set $\{Q_i\}_{i\in \Lambda}$ of key polynomials satisfying the conclusion of Proposition \ref{kp_form}.}

\section{Minimal Pairs and Key Polynomials}\label{kp_mp}
Let $\{Q_i\}_{i\in\Lambda}$ be a complete set of key polynomials for $w$.\\

In this section, we study the relation between the properties of $\{Q_i\}_{i\in\Lambda}$, minimal pairs for a common extension
$\bar{w}$, and the type of $w$ as residue-transcendental, value-transcendental or valuation-algebraic.

\begin{proposition}\label{non_max}
Take an element $i\in\Lambda$. If $i$ is not the maximal element of $\Lambda$ then
\begin{equation}\label{eq:grouprational}
\g_i\subset \g_v\otimes_\Z \Q.
\end{equation}
\end{proposition}
\begin{proof}
By Remark \ref{directconsequence}, it is sufficient to prove (\ref{eq:grouprational}) for $i\in I$. Take an element $i\in I$. Suppose, inductively, that for each $i'\in I$, $i'<i$ we have $\beta_{i'}\in\g_v\otimes_\Z \Q$ (by Remark \ref{directconsequence} this implies that $\beta_j\in\g_v\otimes_\Z \Q$ for each $j\in\Lambda$ with $j<i$). Assume that $\beta_i\notin \g_v\otimes_\Z \Q$, aiming for contradiction. For each element $g\in K[X]$, all the terms of its $Q_i$-expansion have different values, hence $w_i(g)=w(g)$. Since this holds for all $g\in K[X]$, we have $w_i=w$. By Remark \ref{uniqueness}, this is impossible since $i$ is not the maximal element of
$\Lambda$.
\end{proof}

\begin{corollary}\label{vt_kp}
The valuation $w$ is a value-transcendental extension of $v$ if and only if $\Lambda$ has a maximal element $i_0$, that is
$w=w_{i_0}$, $\beta_i\in\g_v\otimes_\Z \Q$ for all $i<i_0$ and $\beta_{i_0}\notin\g_v\otimes_\Z \Q$.
\end{corollary}

\begin{proposition}\label{rt_kp} If $\beta_i\in\g_v\otimes_\Z\Q$ then $w_i$ is a residue-transcendental extension of $v$.
\end{proposition}
\begin{proof}
In what follows, we will adopt the convention that $\g_{-1}:=\g_v$. Let $j\in I$ be such that $i\in\Lambda_j$. Let $l=min\{s\ /\ s\beta_j\in \g_{j-1}\}$. There exists a polynomial $g(X)$ with $\deg(g)<\deg(Q_i)$ such that $w_i(g)=l\beta_i$.\\
Put $t=\left(\frac{Q_i^l}{g}\right)^*$ in $k_{w_i}$. Assume that $t$ is algebraic over $k_v$ and let
$$
t^{n}+a_{n-1}t^{n-1}+\dots+a_0
$$
be an algebraic equation satisfied by $t$ over $k_v$. Choose representatives $a'_{n-1},\dots,a'_0$ in $K$ of the coefficients
$a_{n-1},\dots,a_0$. Then
$w_i\left(\left(\frac{Q_i^l}{g}\right)^{n}+a'_{n-1}\left(\frac{Q_i^l}{g}\right)^{n-1}+\dots+a'_0\right)>0$.  That is,\\
$w_i\left((Q_i^l)^{n}+a'_{n-1}g(Q_i^l)^{n-1}+\dots+a_0g^n\right)>w_i(g^n)= nl\beta_i\geq
w_i\left((Q_i^l)^{n}+a'_{n-1}g(Q_i^l)^{n-1}+\dots+a_0g^n\right)$, which is absurd.
\end{proof}

\begin{corollary}\label{rt_kpresult}
If $i$ is not the maximal element of $\Lambda$ then the valuation $w_i$ is residue-transcendental.
\end{corollary}
\begin{proof}
This is a direct consequence of Proposition \ref{non_max} and Proposition \ref{rt_kp}.
\end{proof}

\begin{proposition}\label{mp_kp}
If $M_{\bar{w}}$ has a maximal element $\delta$, then $\Lambda$ has a maximal element $i_0$ and there exists a root $a\in\Ro(Q_{i_0})$ such that $(a,\delta)$ is a minimal pair of definition for $\bar{w}$.
\end{proposition}
\begin{proof}
Assume that $M_{\bar{w}}$ has a maximal element $\delta$. Let $(a,\delta)$ be a minimal pair of definition for $\bar{w}$ (it exists by Proposition \ref{max_ele}).\\

Let $f(X)$ be the minimal polynomial of $a$ over $K$. The polynomial $f(X)$ is a key polynomial for $w$, since if $g(X)$ is such that $\epsilon(g)\geq\epsilon(f)$ then $\delta(g)\geq\delta (f)=\delta$ hence $\deg(g)\geq \deg(f)$, since $(a,\delta)$ is a minimal pair.\\

If there existed $i\in\Lambda$ such that $\epsilon(Q_i)>\epsilon(f)$, we would have $\delta(Q_i)>\delta(f)=\delta$ which is impossible by definition of $f$ and minimal pair of definition. Hence
\begin{equation}\label{eq:epsilonQileepsilonf}
\epsilon(Q_i)\le\epsilon(f)\quad\text{ for all }i\in\Lambda.
\end{equation}
By definition of key polynomial, this implies that $\deg(Q_i)\le\deg(f)$ for all $i\in\Lambda$.\\

Let $i\in\Lambda$ be such that $w_i(f)=w(f)$. By \cite{J1} Proposition 2.10 (ii) we must have $\epsilon(Q_i)\ge\epsilon(f)$, hence
$\epsilon(Q_i)=\epsilon(f)$ in view of (\ref{eq:epsilonQileepsilonf}). We conclude that $\delta(Q_i)=\delta$.\\

Choose $a'\in\Ro(Q_i)$, such that $\bar{w}(X-a')=\delta$. By Lemma \ref{pairofdefinition} $(a',\delta)$ is a pair of definition for
$\bar{w}$, and since $\mathcal D(\bar{w})=[K(a'):K]$, we have that $(a',\delta)$ is a minimal pair of definition for $\bar{w}$.\\

Now $i$ must be the greatest element of $\Lambda$ since otherwise, if there exists $i'>i$, by Remark \ref{prop_key_pol} 3 we have $e(Q_{i'})>e(Q_i)$ and $\delta(Q_{i'})>\delta(Q_i)=\delta$, which is impossible.
\end{proof}

\begin{theorem}\label{valuation_trans}
The valuation $w$ is a valuation-transcendental extension of $v$ if and only if $\Lambda$ has a maximal element.\\

Moreover, in this case, if $\{Q_i\}_{i\in\Lambda}$ is a complete set of key polynomials for $w$ and $i_0$ the maximal element of
$\Lambda$, then for every common extension $\bar{w}$ of $\bar{v}$ and $w$ to $\bar{K}(X)$, $\mathcal D(\bar{w})=d(w)$ and there exists $a\in \Ro(Q_{i_0})$ such that $w_{(a,\epsilon_{i_0})}=\bar{w}$.
\end{theorem}
\begin{proof}
Let $\bar{w}$ be a common extension of $\bar{v}$ and $w$ to $\bar{K}(X)$. By Proposition \ref{classification} $w$ is
valuation-transcendental if and only if $M_{\bar{w}}$ has a maximal element.\\

Now if $w$ is  valuation-transcendental, by Proposition \ref{mp_kp}, $\Lambda$ has a maximal element.\\

Conversely, if $\Lambda$ has a maximal element, by Corollary \ref{vt_kp} and Proposition \ref{rt_kp}, $w$ is valuation-transcendental. \\

The last statement of the Theorem is a direct consequence of Proposition \ref{mp_kp}.

\end{proof}

\begin{corollary}\label{val_trans_numext}
If $n$ is the number of distinct roots of the final key polynomial in $\{Q_i\}_{i\in\Lambda}$, then there exist at most $n$ common extensions of $w$ and $\bar{v}$ to $\bar{K}(X)$.
\end{corollary}

\begin{example}\label{ExampleX} The pair $(0,\beta_0)=(0,w(X-a))=(0,\epsilon_0)$ is a minimal pair for the valuation $w_0$ (see Remark \ref{prop_key_pol} 2 and the notation that follows it).
\end{example}

\section{Common Extensions}\label{common_ext}

Let $\{Q_i\}_{i\in\Lambda}$ be a complete set of key polynomials for $w$.\\

By Theorem \ref{valuation_trans}, if $w$ is a valuation-transcendental extension of $v$ then $w=w_{i_0}$, where $i_0$ is the maximal element of $\Lambda$, and if $\bar{w}$ is a common extension of $w$ and $\bar{v}$ to $\bar{K}(X)$, then there exists $a\in\Ro(Q_{i_0})$, such that $\bar{w}:=w_{(a,\delta)}$, where $\delta=max (M_{\bar{w}})=\epsilon(Q_{i_0})$.\\

In this section, we investigate which roots $a\in\Ro(Q_{i_0})$ are such that $w_{(a,\delta)}$ is a common extension of $w$ and
$\bar{v}$. By definition of $w_{(a,\delta)}$, it is an extension of $\bar{v}$, hence the question is if the restriction of $w_{(a,\delta)}$ to $K(X)$ is equal to $w$.\\   

Note that for each $i\in\Lambda$, since the valuation $w_i$ is a valuation-transcendental extension of $v$, we have that every common extension of $\bar{v}$ and $w_i$ to $\bar{K}(X)$ has the form $w_{(b,\epsilon_i)}$, with $b\in\Ro(Q_i)$ (Proposition \ref{max_ele}).\\

Assume that we know a minimal pair of definition $(a,\delta)$ for a common extension $\bar{w}$ of $\bar{v}$ and $w$ to $\bar{K}(X)$. The following Lemma gives a criterion to characterize the other minimal pairs of definition for  common extensions of $\bar{v}$ and $w$ to $\bar{K}(X)$.

\begin{lemma}\label{criterion_roots_mp}
Let $\bar{w}$ be a common extension of $\bar{v}$ and $w$ to $\bar{K}(X)$ and let $(a,\delta)$ be a minimal pair for $\bar{w}$. Let $f$ be the minimal polynomial of $a$ over $K$ and let $b\in\Ro(f)$. Then $w_{(b,\delta)}$ is a common extension of $\bar{v}$ and $w$ to $\bar{K}(X)$ if and only if for every $g\in K[X]$ with $\deg(g)<d(w)$, we have $w(g(X))=\bar{v}(g(b))$.
\end{lemma}
\begin{proof}
Suppose first that $w_{(b,\delta)}$  is a common extension of $\bar{v}$ and $w$. By Theorem \ref{valuation_trans} we have
$\mathcal D(w_{(b,\delta)})=d(w)$. By Lemma \ref{minimal_pair}, for every $g\in K[X]$ with $\deg(g)<d(w)$ we have
$\bar{v}(g(b))=w(g(X))$.\\

Conversely, suppose that for every $g\in K[X]$ with $\deg(g)<d(w)$ we have
$$
\bar{v}(g(b))=w(g(X)).
$$
We claim that for every $g\in K[X]$ we have $\bar{v}(g(a)))=\bar{v}(g(b))$.\\

By  Lemma \ref{minimal_pair} we have $\bar{v}(g(a)))=w(g(X))$ for every $g\in K[X]$ with $\deg(g)<d(w)$. Hence for every $g\in
K[X]$ with $\deg(g)<d(w)$, we have $\bar{v}(g(a))=\bar{v}(g(b))$. We still need to prove that for every $g\in K[X]$ with
$\deg(g)\geq d(w)$ we have $\bar{v}(g(a)))=\bar{v}(g(b))$.\\

Consider $g\in K[X]$ with $\deg g\geq d(w)$. Since $(a,\delta)$ is a minimal pair, $d(w)=\deg(f)$. Let $g(X)=q(X)f(X)+r(X)$ be the Euclidean division of $g(X)$ by $f(X)$, with $\deg r(X)<d(w)$.\\
We have $g(a)=r(a)$ and $g(b)=r(b)$. Since we already know that $\bar{v}(r(a))=\bar{v}(r(b))$, the claim is proved.\\

We want to prove that $w_{(b,\delta)}$ is equal to $w$ on $K(X)$. For this it is sufficient to prove that it is equal to $w$ on $K[X]$, therefore it is sufficient to prove that for every $g(X)\in K[X]$ we have $w_{(a,\delta)}(g(X))=w_{(b,\delta)}(g(X))$.\\

Take $g(X)\in K[X]$ and write the Taylor expansions of $g(X)$:\\

\begin{align*}
g(X)&=g_n(a)(X-a)^{n}+\dots+g_0(a),\\
g(X)&=g_n(b)(X-b)^{n}+\dots+g_0(b),
\end{align*}
where for each $t$, $0\leq t\leq n$, we have that $g_t(X)=\partial_t g(X)$ is a polynomial in $K[X]$, hence, by the above discussion,
$\bar{v}(g_t(a))=\bar{v}(g_t(b))$.\\

Now by definition of $w_{(a,\delta)}$ and of $w_{(b,\delta)}$ we have

\begin{align*}
w_{(a,\delta)}(g(X))&=\inf_{0\leq t\leq n}\{\bar{v}(g_t(a))+t\delta\}\\
&=\inf_{0\leq t\leq n}\{\bar{v}(g_t(b))+t\delta\}\\
&=w_{(b,\delta)}(g(X)).
\end{align*}
\end{proof}

\begin{lemma}\label{roots_kp_ineq}
Let $i\in \Lambda$ and suppose that for every $b\in\Ro(Q_i)$, $\bar{w}_i:=w_{(b,\epsilon_i)}$ is an extension of $w_i$. Let $\alpha\in\bar{K}$. If for every root $b\in\Ro(Q_i)$ we have $\bar{v}(\alpha-b)<\epsilon_i$ then
$\bar{v}(Q_i(\alpha))<w_i(Q_i)=\beta_i$.
\end{lemma}
\begin{proof}
Choose $b\in\Ro(Q_i)$ such that $\bar{v}(\alpha-b)\geq \bar{v}(\alpha-c)$ for every $c\in \Ro(Q_i)$. Let
$\bar{w}_i:=w_{(b,\epsilon_i)}$. We will prove that

\begin{equation}\label{roots_kp_ineq_main_eq}
\bar{w}_i(X-c)\geq \bar{v}(\alpha-c)\ \text{ for\ every\ }c\in \Ro(Q_i)
\end{equation}

For every $c\in\Ro(Q_i)$, we have, $\bar{w}_i(X-c)=\inf\{\epsilon_i,\ \bar{v}(b-c)\}$.\\
Now, $\epsilon_i>\bar{v}(\alpha-c)$ by assumption and $\bar{v}(b-c)\geq \inf\{\bar{v}(b-\alpha)$ and
$\bar{v}(\alpha-c)\}\geq \bar{v}(\alpha-c)$, where the last inequality follows from the definition of $b$.\\

We have $Q_i(X)=\prod\limits_l(X-c_l)$, where the $c_l$ are the roots of $Q_i$, which need not be distinct. Now the result follows from the facts that $w_i(Q_i(X))=\bar{w}_i(Q_i(X))$, $\bar{w}_i(X-b)=\epsilon_i>\bar{v}(\alpha-c_l)$ and (\ref{roots_kp_ineq_main_eq}).

\end{proof}

\begin{lemma}\label{roots_kp_eq}
Let $j\in I$ be such that $\vartheta_j=\emptyset$ and for every $b\in\Ro(Q_j)$ the valuation $\bar{w}_{bj}:=w_{(b,\epsilon_j)}$ is an extension of $w_j$. Assume that $\Lambda_{j+1}\ne\emptyset$ and let $a\in \Ro(Q_{j+1})$. If there exists $b\in\Ro(Q_{j})$ such that $\bar{v}(a-b)\geq\epsilon_j$ then $\beta_j=w_j(Q_j)=\bar{v}(Q_j(a))$.
\end{lemma}
\begin{proof}
Choose $b\in\Ro(Q_{j})$ such that $\bar{v}(a-b)\geq \epsilon_j$, and let  $\bar{w}_{bj}:=w_{(b,\epsilon_j)}$. By Lemma \ref{pairofdefinition}, $(a,\epsilon_j)$ is also a pair of definition for $\bar{w}_{bj}$.\\

By Proposition \ref{kp_form} and the comment that follows it, we can write
$$
Q_{j+1}=q_nQ_j^{n}+\dots+q_tQ_j^t+\dots+q_0,\quad\text{ with }q_n=1,\ \deg q_t< d_j
$$
and
\begin{equation}\label{eq:homogeneous}
w(q_t)+t\beta_j=n\beta_j\quad\text{ for }t\in\{0,\dots,n\}.
\end{equation}
By (\ref{eq:homogeneous}), given integers $t_1,t_2$ with $0\le t_1<t_2$, we have $w(q_{t_1})+t_1\beta_j=w(q_{t_2})+t_2\beta_j$, so $w(Q_j(X))=\beta_j=\frac{w(q_{t_2}(X))-w(q_{t_1}(X))}{t_2-t_1}$. Combining this with Lemma \ref{minimal_pair}, we obtain
$$
w(Q_j(X))=\frac{\bar{v}(q_0(a))}n.
$$
Now $0=Q_{j+1}(a)=Q_j^{n}(a)+\dots+q_t(a)Q_j^t(a)+\dots+q_0(a)$. Therefore, there exist $t_1,t_2$, $0\leq t_1<t_2\leq n$ such that $\bar{v}\left(q_{t_1}(a)Q_j^{t_1}(a)\right)=\bar{v}\left(q_{t_2}(a)Q_j^{t_2}(a)\right)$.\\

Thus we have $\bar{v}(Q_j(a))=\frac{\bar{v}(q_{t_2}(a))-\bar{v}(q_{t_1}(a))}{t_2-t_1}=w(Q_j(X))$.

\end{proof}

\begin{lemma}\label{key_pol_roots_lem}
Let $j\in I$ be such that $\vartheta_j=\emptyset$ and for every $b\in\Ro(Q_j)$ the valuation $\bar{w}_{bj}:=w_{(b,\epsilon_j)}$ is an extension of $w_j$. Assume that $\Lambda_{j+1}\ne\emptyset$. Then for every root $a\in \Ro(Q_{j+1})$ there exists $b\in\Ro(Q_j)$ such that $\bar{v}(a-b)\geq\epsilon_j$.
\end{lemma}

\begin{proof}
By Proposition \ref{kp_form} and the comment that follows it, we can write
$$
Q_{j+1}=q_nQ_j^{n}+\dots+q_tQ_j^t+\dots+q_0,\quad\text{ where }q_n=1,\ \deg q_t< d_j\text{ and (\ref{eq:homogeneous}) is satisfied}.
$$
Let $b_1,\dots,b_t$ be the roots of $Q_j$ and $a_1,\dots,a_s$ the roots of $Q_{j+1}$ (not necessarily distinct). We have $s=n\cdot t$ and the resultant of $Q_{j}$ and $Q_{j+1}$ is given by
\begin{equation}
\prod\limits_{\ell=1}^{t}Q_{j+1}(b_\ell)=(-1)^{nt^2}\prod\limits_{k=1}^{s}Q_j(a_k).
\end{equation}

For each $\ell$, $1\leq \ell\leq t$, we have $\bar{v}(Q_{j+1}(b_\ell))=\bar{v}(q_0(b_\ell))=w_j(q_0(X))=n\beta_j$, where the second equality is obtained from Lemma \ref{minimal_pair}. We obtain
 
\begin{equation}
\bar{v}\left(\prod\limits_{k=1}^{s}Q_j(a_k)\right)=\bar{v}\left(\prod\limits_{\ell=1}^{t}Q_{j+1}(b_\ell)\right)=tn\beta_j=s\beta_j.
\end{equation}
Hence
\begin{equation}\label{key_pol_roots_eq}
\sum\limits_{k=1}^{s}\bar{v}(Q_j(a_k))=s\beta_j.
\end{equation}

By Lemmas \ref{roots_kp_ineq} and \ref{roots_kp_eq} we have, for each $k$, $1\leq k\leq s$, $\bar{v}(Q_j(a_k))\leq \beta_j$ (where we apply Lemma \ref{roots_kp_ineq} in the case when for every root $b\in\Ro(Q_i)$ we have $\bar{v}(a_k-b)<\epsilon_i$ and Lemma \ref{roots_kp_eq} in the case when there exists $b\in\Ro(Q_{j})$ such that $\bar{v}(a_k-b)\geq\epsilon_j$). Hence, in view of (\ref{key_pol_roots_eq}) we must have the equality $\bar{v}(Q_j(a_k))= \beta_j$, for each $k$, $1\leq k\leq s$.\\

Therefore, by Lemma \ref{roots_kp_ineq}, for each $k$, $1\leq k\leq s$, there exists $\ell$, $1\leq \ell\leq t$, such that $\bar{v}(a_k-b_\ell)\geq \epsilon_j$.\\
\end{proof}
\begin{proposition}\label{key_pol_roots}
Let $j\in I$ be such that $\vartheta_j=\emptyset$ and consider $b\in\Ro(Q_j)$ such that $\bar{w}_{bj}:=w_{(b,\epsilon_j)}$ is an extension of $w_j$. Assume that $\Lambda_{j+1}\ne\emptyset$, and let $a\in \Ro(Q_{j+1})$ be such that
$\bar{v}(a-b)\geq\epsilon_j$. Then $w_{(a,\epsilon_{j+1})}$ is an extension of $w_{j+1}$.
\end{proposition}
\begin{proof}
Let $\bar{w}_{j+1}$ be a common extension of $\bar{v}$ and $w_{j+1}$ to $\bar{K}(X)$. By Theorem \ref{valuation_trans}, there exists a root $a_1$ of $Q_{j+1}$ such that $\bar{w}_{j+1}=w_{(a_1,\delta)}$. If $a=a_1$, we are done, there is nothing to prove. Otherwise, assume that $a\neq a_1$.\\

In view of Lemma \ref{criterion_roots_mp}, it is sufficient to prove that for all $g\in K[X]$ with $\deg g<d_{j+1}$ we have
$\bar{v}(g(a))=w(g(X))$.\\

By Lemma \ref{pairofdefinition}, $(a,\epsilon_j)$ is also a pair of definition for $\bar{w}_{bj}$. Therefore, by Lemma \ref{minimal_pair}, 
\begin{equation}
\text{for all }g\in K[X]\text{ with }\deg\ g<d_{j}\text{ we have }\bar{v}(g(a))=w_j(g(X))=w(g(X)).\label{eq:equalityfordegg<dj+1}
\end{equation}
By Lemma \ref{roots_kp_eq}, we have
$\bar{v}(Q_j(a))=\beta_j$.\\

We will first prove, by contradiction, that for all $g\in K[X]$ with $\deg g<d_{j+1}$, we have $\bar{v}(g(a))\geq w(g(X))$.\\

Assume that there exists $g\in K[X]$, $\deg g<d_{j+1}$, such that $\bar{v}(g(a))< w(g(X))$. Choose $g$ to be of minimal degree subject to this inequality. By (\ref{eq:equalityfordegg<dj+1}) we must have
\begin{equation}
\deg g\geq d_j.\label{eq:degreelessorequal}
\end{equation}
Let $g(X)=q(X)Q_j(X)+r(X)$ be the Euclidean division of $g(X)$ by $Q_j(X)$, with
$$
\deg r< d_j.
$$
Note that
\begin{equation}
\deg q<\deg g.\label{eq:degq<degg}
\end{equation}
We have 
\begin{align*}
\bar{v}(g(a))&\geq \inf\left\{\bar{v}(q(a))+\bar{v}(Q_j(a)),\ \bar{v}(r(a))\right\}\\
&\geq \inf\left\{w_j(q(X))+w_j(Q_j(X)),\ w_j(r(X))\right\}\\
&=w_j(g(X))\\
&=w(g(X)),
\end{align*}
where the second inequality follows from (\ref{eq:equalityfordegg<dj+1}), (\ref{eq:degq<degg}), the assumed minimality of $\deg\ g$ and the fact that $(a,\epsilon_j)$ is a pair of definition for $\bar{w}_{bj}$ which extends $w_j$. The above inequality contradicts our assumption.

Again, aiming for contradiction, we will assume that there exists
$$
g\in K[X],\quad d_j\leq \deg g<d_{j+1}
$$
such that $\bar{v}(g(a))>w(g(X))$. Assume that $g$ is chosen of minimal degree subject to this inequality.

Let $Q_{j+1}=Qg+R$ be the Euclidean division of $Q_{j+1}(X)$ by $g(X)$, with $\deg R< \deg g$.\\

We have $w(Q_{j+1})>w(Qg)=w(R)$, since $w_j(Qg)=w(Qg)$ and $w_j(R)=w(R)$ but $w(Q_{j+1})>w_j(Q_{j+1})$.

Moreover, we have $0=Q_{j+1}(a)=Q(a)g(a)+R(a)$. Therefore, we must have
$$
\bar{v}(R(a))=\bar{v}(Q(a)g(a)).
$$
But $\bar{v}(R(a))=w(R(X))$ by the assumed minimality of the degree of $g$. Hence we must have $w(Q(X)g(X))=\bar{v}(Q(a)g(a))$, but $w(Q(X))\leq \bar{v}(Q(a))$ and $w(g(X))<\bar{v}(g(a))$ which is impossible.
\end{proof}

As a direct consequence of Lemma \ref{key_pol_roots_lem} and Proposition \ref{key_pol_roots} we have the following Theorem:
\begin{theorem}\label{key_pol_all_roots}
Let $j\in I$ be such that $\vartheta_j=\emptyset$, and for every $b\in\Ro(Q_j)$, $\bar{w}_{bj}:=w_{(b,\epsilon_j)}$ is an extension of $w_j$. Assume that $\Lambda_{j+1}\ne\emptyset$. Then for every root $a\in \Ro(Q_{j+1})$, $w_{(a,\epsilon_{j+1})}$ is an extension of $w_{j+1}$.
\end{theorem}

\begin{corollary}\label{allroots}
Let $j_0\in I$ be such that $w=w_{j_0}$ and assume that for every $j\in I$ we have $\vartheta_j=\emptyset$. Then for every root $a\in \Ro(Q_{j_0})$ the valuation $w_{(a,\epsilon_{j_0})}$ is an extension of $w$.
\end{corollary}
\begin{proof}
We use induction on $j\leq j_0$. The base of the induction is nothing but Example \ref{ExampleX}. The induction step is given by Theorem \ref{key_pol_all_roots}.
\end{proof}


\begin{thebibliography}{1}

\bibitem{AZ} V. Alexandru, A. Zaharescu, {\em A theorem of characterization of residual transcendental extensions of a valuation}, J. Math. Kyoto Univ. 28 (1988), 579-592.

\bibitem{APZ} V. Alexandru, N. Popescu, A. Zaharescu, {\em Minimal pairs of definition of a residual transcendental extension of a valuation}, J. Math. Kyoto Univ. 30 (1990), 207-225.

\bibitem{Na2} M. Alberich-Carraminana, A. F. Boix, J. Fern\'andez, J. Gu\`ardia, E. Nart, J. Ro\'e {\em Invariants of limit key polynomials.}, arXiv:2005.04406v1 {[}math.AG{]}. 

\bibitem{CMT} S. D. Cutkosky,  H. Mourtada, B. Teissier  {\em On the construction of valuations and generating sequences on hypersurface singularities.},  arXiv:1904.10702v1 {[}math.AG{]}.

\bibitem{DMS} J. Decaup, W. Mahboub, M. Spivakovsky, Abstract key polynomials and comparison theorems with the
key polynomials of MacLane-Vaqui\'e, Illinois J. Math. 62, Number 1-4 (2018), 253-270.

\bibitem{GKZ} I. M. Gelfand, M. M. Kapranov, A. V. Zelevinsky, (1994), Discriminants, resultants, and multidimensional determinants, Boston: Birkh\"auser, ISBN 978-0-8176-3660-9

\bibitem{HOS} F. J. Herrera Govantes, M. A. Olalla Acosta, M. Spivakovsky, {\em Valuations in algebraic field extensions},
Journal of Algebra, Volume 312, Issue 2 (2007), pages 1033-1074.

\bibitem{HMOS} F. J. Herrera Govantes, W. Mahboub, M. A. Olalla Acosta, M. Spivakovsky {\em Key polynomials for simple
extensions of valued fields.}, arXiv:1406.0657v3 {[}math.AG{]}. 

\bibitem{Kas} O. Kashcheyeva, {\em Constructing examples of semigroups of valuations}, J. Pure Appl. Algebra 220 (2016), 3826-3860

\bibitem{Ku} F.-V. Kuhlmann, {\em Value groups, residue fields and bad places of rational function fields}, Trans. Amer. Math. Soc. 356 (2004), 4559-4600.

\bibitem{ML1} S. MacLane, {\em A construction for prime ideals as absolute values of an algebraic field}, Duke Math. J., vol 2 (1936),  492-510. 

\bibitem{ML2} S. MacLane {\em A construction for absolute values in polynomial rings}, Transactions of the AMS, vol 40 (1936),
 363-395.

\bibitem{Ma} W. Mahboub {\em Key Polynomials}, Journal of Pure and Applied Algebra, 217(6) (2013), pages 989-1006.

\bibitem{Na1} E. Nart, Key polynomials over valued fields, Publ. Mat. 64 (2020), 3-42.

\bibitem{J1} J. Novacoski and M. Spivakovsky {\em  Key polynomials and pseudo-convergent sequences}, Journal of Algebra 495 (2018), 199-219.

\bibitem{J2} J. Novacoski {\em  Key polynomials and minimal pairs}, Journal of Algebra 523 (2019), 1-14.

\bibitem{San} J.-C. San Saturnino, Defect of an extension, {\em key polynomials and local uniformization}, Journal of Algebra 481 (2017), 91-119.

\bibitem{San2} J.-C. San Saturnino, {\em Ramification theory and key polynomials.},  arXiv:1602.08284v1 {[}math.AG{]}. 

\bibitem{V0} M. Vaqui\'e, {\em Famille admise associ\'ee \`a une valuation de $K[x]$. {[}Admissible family associated with a
valuation of $K[x]${]}}, Singularit\'es Franco-Japonaises, 391-428, S\'emin. Congr., 10, Soc. Math. France, Paris, 2005. 

\bibitem{V} M. Vaqui\'e, {\em Extension d'une valuation. {[}Extension of a valuation{]}}, Trans. Amer. Math. Soc. 359 (2007),
no. 7, 3439--3481. (electronic)

\bibitem{V1}  M. Vaqui\'e, {\em Famille admissible de valuations et d\'efaut d'une extension. {[}Admissible family of valuations and defect of an extension{]}}, J. Algebra 311 (2007), no. 2, 859--876. 

\bibitem{V2}  M. Vaqui\'e, {\em Extensions de valuation et polygone de Newton. {[}Valuation extensions and Newton polygon{]}}
Ann. Inst. Fourier (Grenoble) 58 (2008), no. 7, 2503--2541.

\bibitem{V3}  M. Vaqui\'e, {\em Valuation augment\'ee et paire minimale}, arXiv:2005.03298v1 {[}math.CA{]}. 

\end{thebibliography}
\end{document}